\documentclass[11pt]{article}
\usepackage{latexsym}
\usepackage{graphicx} 
\usepackage{color} 
\usepackage{amsmath,amssymb, amsthm}
\usepackage{enumerate}
\usepackage{float}
\usepackage{chemarr}
\usepackage{url}
\usepackage{stackrel}

\newtheorem{theorem}{Theorem}[section]
\newtheorem{corollary}[theorem]{Corollary}
\newtheorem{proposition}[theorem]{Proposition}
\newtheorem{lemma}[theorem]{Lemma}
\newtheorem{example}[theorem]{Example}
\newtheorem{notation}[theorem]{Notation}
\newtheorem{definition}[theorem]{Definition}

\newtheorem{algorithm}[theorem]{Algorithm}
\newtheorem{remark}[theorem]{Remark}

\newcommand{\be}{\begin{equation}}
\newcommand{\ee}{\end{equation}}
\newcommand{\bd}{\begin{displaymath}}
\newcommand{\ed}{\end{displaymath}}
\newcommand{\beal}{\begin{align}}
\newcommand{\enal}{\end{align}}
\newcommand{\been}{\begin{enumerate}}
\newcommand{\enen}{\end{enumerate}}
\newcommand{\beit}{\begin{itemize}}
\newcommand{\enit}{\end{itemize}}

\newcommand{\CRN}{chemical reaction network }
\newcommand{\CRNNoSpace}{chemical reaction network}
\newcommand{\CRNs}{chemical reaction networks }

\newcommand{\op}{\operatorname}
\newcommand{\sen}{square embedded network }
\newcommand{\sens}{square embedded networks }
\newcommand{\Nred}{N^{\rm red} }
\newcommand{\Net}{\mathfrak{G}}

\def\SS{\mathcal S}
\def\CC{\mathcal C}
\def\RR{\mathcal R}

\newcommand{\ra}{\rightarrow}
\newcommand{\lra}{\leftrightarrows}
\newcommand{\la}{\leftarrow}

\newcommand{\R}{\mathbb{R}}

\newcommand{\Z}{\mathbb{Z}}

\newcommand{\sign}{\operatorname{sign}}
\newcommand{\Jac}{\operatorname{Jac}}

\newcommand{\Or}{\operatorname{Or}}
\newcommand{\Gred}{G^{\rm red}}

\begin{document}
\title{Simplifying the Jacobian Criterion for precluding multistationarity in chemical reaction networks}
\author{Badal Joshi and Anne Shiu}
\date{February 14, 2012}

\maketitle

\begin{abstract}
Chemical reaction networks taken with mass-action kinetics are dynamical systems that arise in chemical engineering and systems biology.  In general, determining whether a chemical reaction network admits multiple steady states is difficult, as this requires determining existence of multiple positive solutions to a large system of polynomials with unknown coefficients. However, in certain cases, various easy criteria can be applied. One such test is the Jacobian Criterion, due to Craciun and Feinberg, which gives sufficient conditions for ruling out the possibility of multiple steady states.  A chemical reaction network is said to pass the Jacobian Criterion if all terms in the determinant expansion of its parametrized Jacobian matrix have the same sign.  In this article, we present a procedure which simplifies the application of the Jacobian Criterion, and as a result, we identify a new class of networks for which multiple steady states is precluded: those in which all chemical species have total molecularity of at most two.  The total molecularity of a species refers to the sum of all of its stoichiometric coefficients in the network. We illustrate our results by examining enzyme catalysis networks.
\end{abstract}

\section{Introduction}
This article concerns an important class of dynamical systems arising in chemical engineering and systems biology, namely, chemical reaction networks taken with mass-action kinetics.  As bistable chemical systems are thought to be the underpinnings of biochemical switches, a key question is to determine which systems admit multiple steady states.  However, the systems of interest are typically high-dimensional, nonlinear, and parametrized by many unknown reaction rate constants.  Therefore, determining whether a chemical reaction network admits multiple steady states is difficult: it requires determining existence of multiple positive solutions to a large system of polynomials with unknown coefficients. However, various criteria have been developed that often can answer this question more easily.  For example, the Deficiency and Advanced Deficiency Theories developed by Ellison, Feinberg, Horn, and Jackson in many cases can affirm that a network admits multiple steady states or can rule out the possibility \cite{EllisonThesis, FeinLectures, FeinDefZeroOne, HornJackson72}.  These results have been implemented in the CRNT Toolbox, computer software developed by Feinberg and improved by Ellison and Ji, which is freely available \cite{Toolbox}.  Related software programs include BioNetX \cite{BioNetX,Pantea_comput} and Chemical Reaction Network Software for Mathematica \cite{CRNsoftware}.  For systems for which the software approaches are inconclusive, Conradi~{\em et~al.}\ advocate a two-pronged approach: they first determine whether certain subnetworks admit multiple steady states, and, if so, test whether these instances can be lifted to the original network \cite{Conradi_subnetwork}.  Another criterion for lifting multistationarity from subnetworks appears in \cite{JS2}.  For systems whose steady state locus is defined by binomials, a criterion for multistationarity is given by P\'erez Mill\'an~{\em et~al.}\ in~\cite{TSS}.  Finally, a degree theory method which can establish multiple steady states can be found in work of Craciun, Helton, and Williams \cite{CHW08}.

A test for answering the question of multistationarity, which is implemented in some of the above-mentioned software programs, is the Jacobian Criterion, which is the focus of this article.  The criterion, which is due to Craciun and Feinberg, gives sufficient conditions for precluding multiple steady states \cite{ME1,ME_entrapped,ME2,ME3}. This work has been extended by Banaji and Craciun \cite{BanajiCraciun2009, BanajiCraciun2010} and by Feliu and Wiuf \cite{FW_prec}. 
A chemical reaction network is said to pass the Jacobian Criterion (or to be `injective') if all terms in the determinant expansion of its parametrized Jacobian matrix have the same sign (see Lemma~\ref{lem:CracFein_subnetwork}).  
We prove that for a given network the Jacobian Criterion can be checked by examining certain subnetworks (Theorem~\ref{thm:proper_subnetworks}).

Our main result (Algorithm~\ref{algor:JC}) is a procedure which simplifies the application of the criterion in two steps: first, the network is reduced to a simpler network, and then the number of 
subnetworks that must be examined is decreased.  As an important consequence, {\em we identify a new class of networks for which multiple steady states is precluded: those networks in which each chemical species has total molecularity of less than or equal to two} (Theorem~\ref{thm:total_mol_2}).  The total molecularity of a species refers to how many non-flow reactions it appears in, where reversible reactions are counted only once and each occurrence of the species is counted with its stoichiometric coefficient; see Definition~\ref{def:reversible_and_TM}.  More generally, Algorithm~\ref{algor:JC} allows us in certain cases to rule out the possibility of multiple steady states simply by inspecting the network.  In addition, the reduction in the number of subnetworks that must be examined lends support to the observation that typically very few terms (if any) in the Jacobian determinant have anomalous sign~\cite{ME1}; this topic is explored by Helton, Klep, and Gomez~\cite{HeltonDeterminant} and Craciun, Helton, and Williams~\cite{CHW08}.  A proposed application of our work is to the problem of enumerating networks that admit multiple steady states: our algorithm quickly removes from consideration many small networks that are computationally expensive to enumerate (see Remark~\ref{rem:enum}).  
Finally, we provide examples which demonstrate that our approach is complementary to that of the Species-Reaction Graph Theorem of Craciun and Feinberg \cite{ME2,CTF06}.


We note that our results are stated for {\em continuous-flow stirred-tank reactors (CFSTRs)}, chemical reaction networks in which all chemical species are removed at rates proportional to their concentrations (see Definition~\ref{def:reversible_and_TM}).  
However, even for more general networks in which not all species take part in 
the outflows, our results often nevertheless apply in one of the following two ways.  First, one can augment a network by any missing 
outflows so that it becomes a CFSTR, and then apply our results to determine whether it passes the Jacobian Criterion.  Then, under mild assumptions on the original network (such as weak-reversibility; see Definition~\ref{def:reversible_and_TM}), the recent work of Craciun and Feinberg allows us to conclude that the original network also precludes multiple steady states \cite[Corollary~8.3]{ME3}.  A second approach is due to work of Feliu and Wiuf.  Their extended definition of the Jacobian Criterion applies to all networks \cite{FW_prec}, thereby bypassing the need to augment a network by missing outflows.  Finally, we emphasize that the Jacobian Criterion applies for more general dynamics than mass-action kinetics \cite{BanajiCraciun2009,CHW08}; our results also may be used in those settings.

This article is organized as follows.  
Section~\ref{sec:intro_CRN} provides an introduction to chemical reaction systems.  
Section~\ref{sec:JC} describes the Jacobian Criterion, and Section~\ref{sec:analysis} lays the groundwork for simplifying its implementation.  
The main result of Section~\ref{sec:TM2} (Theorem~\ref{thm:total_mol_2}) precludes multistationarity from any network in which each chemical species has total molecularity at most two.
Section~\ref{sec:algor} presents Algorithm~\ref{algor:JC}, our procedure which simplifies the application of the Jacobian Criterion.
Finally, Section~\ref{sec:examples} contains examples such as enzyme catalysis networks which illustrate our main results.

\section{Chemical reaction network theory} \label{sec:intro_CRN}
In this section we review the standard notation and define `total molecularity.'    
An example of a {\em chemical reaction} is denoted by the following:
\begin{align*}
  2X_{1}+X_{2} ~\rightarrow~ X_{3}~.
\end{align*}
The $X_{i}$ are called chemical {\em species}, and $2X_{1}+X_{2}$ and
$X_{3}$ are called chemical {\em complexes.}  Assigning the {\em
  reactant} complex $2X_{1}+X_{2}$ to the vector $y =
(2,1,0)$ and the {\em product} complex $X_{3}$ to the vector
$y'=(0,0,1)$, we rewrite the reaction as $ y \rightarrow y'$. In
general we let $s$ denote the total number of species $X_{i}$, and we
consider a set of $m$ reactions, each denoted by
\begin{align*}
  y_{k} \rightarrow y_{k}'~, 
\end{align*} 
for $k \in \{1,2,\dots,m\}$, and $y_k, y_k' \in \Z^s_{\ge 0}$, with $y_k \ne
y_k'$.  
We write each such vector $y_k$  as $y_k = \left( y_{k1}, y_{k2}, \dots, y_{ks}\right) \in \Z^s_{\geq 0}$,
and we call $y_{ki}$ the {\em stoichiometric coefficient} of species $X_i$ in 
complex $y_k$.  
For ease of notation, when there is no need for
enumeration we will drop the subscript $k$ from the notation
for the complexes and reactions.  

A {\em self-catalyzing reaction} will refer to a reaction, such as $A+B \ra 2A+C$, in which there is a species (here, $A$) that is a {\em self-catalyst}: it appears with a higher stoichiometric coefficient in the product complex than in the reactant complex. 
\begin{definition}   \label{def:crn}
  Let $\SS = \{X_i\}$, $\CC = \{y\},$ and $\RR = \{y \to y'\}$ denote finite 
  sets of species, complexes, and reactions, respectively.  The triple
  $\{\SS, \CC, \RR \}$ is called a {\em chemical reaction network} if it satisfies the following:
  \been
  	\item for each complex $y \in \CC$, there exists a reaction in $\RR$ for which $y$ is the reactant complex or $y$ is the product complex, and
	\item for each species $X_i \in \SS$, there exists a complex $y \in \CC$ that contains $X_i$.
  \enen
\end{definition}

We say that a species $X_i$ is a {\em reactant species} (respectively, {\em product species}) if it appears in a reactant (respectively, product) complex of some reaction. A subset of the reactions $\RR' \subset \RR$ defines the {\em subnetwork} $\{\SS|_{\CC|_{\RR'}},\CC|_{\RR'},\RR' \}$, where $\CC|_{\RR'}$ denotes the set of complexes that appear in the reactions $\RR'$, and $\SS|_{\CC_{\RR'}}$ denotes the set of species that appear in those complexes. A network {\em decouples} if there exist nonempty sets of reactions $\RR'$ and $\RR''$ that contain disjoint sets of species and such that $\RR = \RR' ~ \dot \cup ~ \RR''$. Otherwise, we say that the network is {\em coupled}. 
\begin{definition} \label{def:reversible_and_TM}
\begin{enumerate}
	\item  A \CRN is {\em weakly-reversible} if for each reaction $y \ra y' \in \RR$, there exists a sequence of reactions in $\RR$ from $y'$ to $y$ (that is, reactions $y' \ra y_1  \ra y_2 \ra \cdots \ra y_k \ra y$).
  A reaction $y \to y' \in
  \RR$ is {\em reversible} if its reverse reaction $y' \to y$ is also a reaction in $\RR$. 
	\item 
	A {\em flow reaction} contains only one molecule; such a reaction is either an {\em inflow reaction} $0 \ra X_i$ or an {\em outflow reaction} $X_i \ra 0$.  A {\em non-flow reaction} is any reaction that is not a flow reaction.
	\item   
  A \CRN is a {\em continuous-flow stirred-tank reactor network} {\em (CFSTR network)} if it contains all outflow reactions $X_i \ra 0$ (for all $X_i \in \SS$). 
 	\item 
Consider a \CRN whose set of non-flow reactions is given by:
	\begin{align*}
	y_1 & \ra y_1' \quad & y_2 & \ra y_2' \quad & \ldots & \quad & y_l & \ra y_l'  \\
	y_{l+1} & \lra y_{l+1}' \quad & y_{l+2} & \lra y_{l+2}' \quad & \ldots & \quad & y_{l+k} & \lra y_{l+k}' ~,
	\end{align*}
where none of the first $l$ reactions is reversible.  The {\bf total molecularity} of a species $X_i$ in the network, denoted by $\op{TM}(X_i)$, is the following integer:
	\begin{align*}
	\op{TM}(X_i) \quad = \quad \sum_{j=1}^{l+k} \left( y_{ji} + y_{ji}' \right)~,
	\end{align*}	
	where $y_{ji}$ is the stoichiometric coefficient of species $X_i$ in the complex $y_j$.
  \end{enumerate}
\end{definition}
\noindent
Note that Craciun and Feinberg use the term ``feed reactions'' for inflow reactions and ``true reactions'' for non-flow reactions.
In chemical engineering, a CFSTR refers to a tank in which reactions occur.  An inflow reaction represents the flow of species (at a constant rate) into the tank in which the non-flow reactions take place, and an outflow reaction represent the removal of a species (at rate proportional to its concentration).  

\subsection{Dynamics}
The concentration vector $x(t) = \left(x_1(t), x_2(t), \ldots, x_{|\SS|}(t) \right)$ will track the concentration $x_i(t)$ of the species $X_i$ at time $t$.  
For a \CRN $\left\{ \SS, \CC, \RR=\{y_k \ra y_k' \} \right\}$, {\em mass-action kinetics} defines the following system of ordinary differential equations:
\begin{equation}  \label{eq:main}
  \dot x(t) \quad = \quad \sum_{k=1}^{|\RR | } \kappa_k x(t)^{y_k}(y_k' - y_k) \quad =: \quad f(x(t))~,
\end{equation}
where the last equality is a definition.  Here, we make use of multi-index notation: $x(t)^{y_k}:= x_1(t)^{y_{k1}} x_2(t)^{y_{k2}} \cdots x_s(t)^{y_{ks}} $.  (By convention, $x_i^0 := 1$).  
Also, the $\kappa_k$'s 
denote unknown (positive) {\em reaction rate constants}. 
A concentration vector $\overline x \in \R^s_{> 0}$ is a (positive) {\em steady state} of the mass-action system \eqref{eq:main} if $f(\overline x) = 0$.

\begin{remark}
We can scale each concentration $x_i$ so that each outflow reaction (a reaction $y_k \ra y_k'$ with $y_k=X_j$ and $y_k'=0$) has outflow rate equal to one ($\kappa_k = 1$ for such a reaction) \cite[Remark 2.9]{ME1}.  We will make this assumption beginning in the next example.
\end{remark}

In the following example and all others in this article, we will label species by distinct letters such as $A,B,\dots$ rather than $X_1,X_2,\dots$.
\begin{example} \label{ex:diffEq}
Consider the following CFSTR:
	\begin{align*}
	0 \stackrel[k_2=1]{k_1}{\rightleftarrows} A \quad \quad 
	0 \stackrel[k_4=1]{k_3}{\rightleftarrows} B \quad \quad 
	A+B \stackrel{k_{5}}{\rightarrow} 2A
	\end{align*}
The total molecularity of species $A$ and $B$ are three and one, respectively.  
The mass-action differential equations~\eqref{eq:main} for this network are the following:
	\begin{align*}
	\frac{d}{dt} x_A \quad & = \quad k_1 - x_A + k_5 x_A x_B \\ 
	\frac{d}{dt} x_B \quad & = \quad k_3 - x_B - k_5 x_A x_B  ~.
	\end{align*}
\end{example}

\begin{definition}
A CFSTR is said to admit {\em multiple steady states} or is {\em multistationary} if there exist reaction rate constants $\kappa_{i} \in \mathbb{R}_{> 0}$ 
whose resulting system~\eqref{eq:main} admits two or more positive steady states.
\end{definition}

\subsection{Square networks}
Throughout this article, we will be interested in subnetworks in which the number of species and the number of reactions are equal.  We will call such subnetworks `square.'

\begin{definition} \label{def:orientation}
\been
	\item A \CRN is said to be $n$-{\em square} if it contains exactly $n$ species and exactly $n$ directed reactions.  A network is said to be {\em square} if it is $n$-{\em square} for some $n \in \Z_{\geq 0}$.  (A $0$-square network is the empty network.)
	\item For an $n$-square \CRN $N = \{ \SS, \CC, \RR\}$ with reaction set $\RR=\{ y_i \rightarrow y_i' ~|~ i=1,2,\dots, n \}$, the  {\em reactant matrix} and {\em reaction matrix} are the integer $(n \times n)$-matrices
\[ 
M_N ~:=~
	\left[\begin{array}{c}
    	y_1 \\
    	y_2  \\
    	\vdots \\
	y_n
    \end{array}\right]
      \quad \quad {\rm and} \quad \quad
R_N ~:=~
	\left[\begin{array}{c}
    	y_1 - y_1' \\
    	y_2 - y_2' \\
    	\vdots \\
	y_n -y_n'
    \end{array}\right]~, \quad {\rm respectively.}
\]
	\item For a square \CRN $N$, we define its  {\em orientation} $\op{Or}(N) \in \{-1,0,1\}$ as follows: 
	\begin{align*}
	\operatorname{Or}(N) \quad := \quad \sign \left( \det(M_N) \det (R_N) \right)~.
	\end{align*}
\enen
\end{definition}
As defined, our reactant and reaction matrices are the transposes of those of Craciun and Feinberg \cite{ME1}.  However, the orientation of a network is unaffected by transposing either matrix.  
Note that a reaction $y_i \rightarrow y_i'$ is encoded by the vector $y_i - y_i'$ in the reaction matrix $R_N$, which is the negative of its reaction vector $y_i' - y_i$ which appears in the mass-action differential equations~\eqref{eq:main}.

\section{The Jacobian Criterion} \label{sec:JC}
The Jacobian Criterion concerns the determinant expansion of the Jacobian matrix of the differential equations~\eqref{eq:main} of a \CRNNoSpace.
\begin{definition}  Consider a \CRN $\Net$ with $s$ species.  
The {\em Jacobian matrix} associated to $\Net$ is the $s\times s$-matrix whose $(i,j)$-th entry is the partial derivative $\frac{\partial}{\partial x_j}f_i(x)$, where $f_i$ is the $i$-th differential equation in~\eqref{eq:main}.  The {\em Jacobian determinant} of $\Net$, denoted by $\Jac(\Net)$, refers to the determinant of the negative of the Jacobian matrix, which is a polynomial in the variables $x_i$ (for $i=1,2,\dots s$) and $\kappa_i$. 
\end{definition}

Note that Craciun and Feinberg use the notation $\det \left( \frac{\partial p_{\Net}}{\partial x} (x,\kappa)  \right)$ for $\Jac(\Net))$ \cite{ME1}; here $p_{\Net}$ denotes the negative of the mass-action differential equations for the network obtained from $\Net$ by removing all inflow reactions.

\begin{example} \label{ex:Jac}
We now return to the CFSTR examined in Example~\ref{ex:diffEq}, which we denote by $\Net$.  Its Jacobian matrix is:
	\begin{align*}
		\left( 
		\begin{array}{cc}
		-1 + k_5 x_B & k_5 x_A  \\
		-k_5 x_B & -1 - k_5 x_A \\
		\end{array} 
		\right)~. 
	\end{align*}
Therefore, the Jacobian determinant of $\Net$ is the following polynomial:
	\begin{align} \label{eq:ex_Jac}
	\Jac ( \Net ) \quad = \quad  \det 
		\left( 
		\begin{array}{cc}
		1 - k_5 x_B & - k_5 x_A  \\
		k_5 x_B & 1 + k_5 x_A \\
		\end{array} 
		\right) \quad = \quad k_5 x_A - k_5 x_B +1 ~.
	\end{align}
\end{example}

The following result, which is due to Craciun and Feinberg \cite{ME1}, states that the determinant expansion of the Jacobian matrix has at most one term in the case of a square \CRNNoSpace.

\begin{lemma} [Theorem~3.2 of~\cite{ME1}] \label{lem:CracFein_subnetwork0}
Let $N$ be a \CRN that contains $s$ species and $s$ directed reactions. Then the Jacobian determinant of $N$, ${\Jac}(N)$, is either zero or a monomial (in the variables $x_i$ and $\kappa_i$) whose coefficient is given by the following product:
	\begin{align*}
	\det (M_{N}) \cdot \det (R_{N})~,
	\end{align*}
where $M_N$ is the reactant matrix and $R_N$ is the reaction matrix.
\end{lemma}

The following result is due to Craciun and Feinberg.
\begin{lemma} [Theorems~3.1 and~3.3 \cite{ME1}] \label{lem:CracFein_subnetwork}
Let $\Net=\Net^{s,r,r_o,r_i}$ be a \CRN that contains $s$ species, $r$ non-flow reactions, $r_o$ outflow reactions, and $r_i$ inflow reactions.  
 Let $\Jac(\Net)$ denote the Jacobian determinant associated with $\Net$, which is a polynomial in variables $x_1, x_2, \dots, x_s,$ $\kappa_1, \kappa_2, \dots, \kappa_{r}$. Let $G_1,G_2,\ldots,G_{{r+r_o+r_i} \choose s}$ enumerate the subnetworks of $\Net$ that contain exactly $s$ reactions, such that the first ${r+r_o} \choose {s}$ do not contain inflow reactions.
\been
\item Each term in the expansion of $\Jac(\Net)$ is the Jacobian determinant of a unique subnetwork $G_i$.  Conversely, the Jacobian determinant of each subnetwork $G_i$ is either zero or is a monomial in the Jacobian determinant of $\Net$.  
In other words,
\bd
\Jac(\Net) \quad = \quad \sum_{i=1}^{{r+r_o+r_i} \choose s} \Jac(G_i) 
	\quad = \quad \sum_{i=1}^{{r+r_o} \choose s} \Jac(G_i) ~.\ed
\item Moreover, if $\Net$ is a CFSTR, then the following are equivalent:
\begin{itemize}
	\item The polynomial $\Jac(\Net)$ is positive for all $x \in \R^s_{>0}$ and $\kappa_i>0$.
	\item Each term $\Jac(G_i)$ is either zero or a positive monomial.
	\item Each subnetwork $G_i$ has non-negative orientation: $\Or(G_i) \geq 0$.
\end{itemize}
If these equivalent conditions hold, then $\Net$ does not admit multiple steady states.  
 \enen 
\end{lemma}
One of the main results in the next section is Theorem~\ref{thm:proper_subnetworks}, which will extend the equivalent conditions that appear in  part~2 of Lemma~\ref{lem:CracFein_subnetwork} so that instead of checking whether the subnetworks $G_i$ have non-negative orientation, we need only check the orientations of certain smaller `embedded' networks (Definition \ref{def:embeddednet}).  If any of these conditions hold for a CFSTR $\Net$, for instance if $\Jac(\Net)$ has only positive terms, then we say that $\Net$ {\em passes the Jacobian Criterion} (or $\Net$ is `injective').  

\begin{remark}
 $\Jac(\Net)$ always has at least one positive term when $\Net$ contains all outflow reactions (such as for a CFSTR).  This arises because one of the subnetworks $G_i$ consists of the $s$ outflow reactions.  Clearly, the reactant and reaction matrices of this subnetwork are both the identity matrix, so its orientation is indeed positive. 
\end{remark}
The aim of this article is to simplify the computation that must be performed in determining whether a CFSTR passes the Jacobian Criterion, that is, whether all other terms in the determinant expansion are positive.  In fact, there are typically few negative terms \cite{ME1}; this topic is explored in \cite{CHW08,HeltonDeterminant}. 

\begin{notation}
Although Feliu and Wiuf have extended the definition of the Jacobian Criterion to apply also to non-CFSTR networks \cite{FW_prec}, the main focus of this article is the case when the \CRN $\Net$ is a CFSTR. The subnetwork of $\Net$ containing all the non-flow reactions but none of the flow reactions will be called {\em the non-flow subnetwork of $\Net$} and denoted by $G$.  Conversely, any reaction network $N=\{\SS, \CC, \RR \}$ is contained in a CFSTR obtained by including all outflow reactions:  we denote this CFSTR by $$\tilde N := \{\SS,~\CC \cup \SS \cup \{0\},~\RR \cup \{X_i \ra 0\}_{X_i \in \SS} \}~.$$
\end{notation}
\begin{example} \label{ex:1}
We now return to the CFSTR examined in Examples~\ref{ex:diffEq} and~\ref{ex:Jac}, which has $s=2$ species and $r=1$ non-flow reaction.  Its non-flow subnetwork is:
	\begin{align*}
	G \quad := 
		\quad \{A+B \ra 2A\}~.
	\end{align*}

The ${r+r_o \choose s} = {3 \choose 2} =3$ $2$-square subnetworks of $\tilde G = \Net = \{A \lra 0,~ B \lra 0,~ A+B \ra 2A\}$ that do not contain inflow reactions are the following:
	\begin{align*}
	G_1 \quad & = \quad \{A \rightarrow 0,~ B \rightarrow 0\}~, \\
	G_2 \quad & = \quad \{A \rightarrow 0,~ A+B \rightarrow 2A\}, \quad {\rm and} \\
	G_3 \quad & = \quad \{B \rightarrow 0,~ A+B \rightarrow 2A\}~.
	\end{align*}
As noted earlier, the subnetwork comprised of all outflow reactions, $G_1$, has positive orientation: $\Or(G_1)=1$.  Now we examine the reactant and reaction matrices of the remaining subnetworks $G_2$ and $G_3$:
	\begin{align} \label{matrices_ex}
	M_{G_2} \quad & = \quad 
		\left( 
		\begin{array}{cc}
		1 & 0  \\
		1 & 1  \\
		\end{array} 
		\right) \quad 
	& R_{G_2} \quad &= \quad 
		\left( 
		\begin{array}{cc}
		1 & 0  \\
		-1 & 1  \\
		\end{array} 
		\right) \quad \\  
	M_{G_3} \quad & = \quad 
		\left( 
		\begin{array}{cc}
		0 & 1  \\
		1 & 1  \\
		\end{array} 
		\right) \quad 
	& R_{G_3} \quad &= \quad 
		\left( 
		\begin{array}{cc}
		0 & 1  \\
		-1 & 1  \\
		\end{array} 
		\right)~. \notag
	\end{align}
We see that the orientations of the two networks are $\Or(G_2)=$ $\sign \left( \det(M_{G_2}) \det (R_{G_2}) \right)$ $=1$ and $\Or(G_3)=\sign \left( \det(M_{G_3}) \det (R_{G_3}) \right)=-1$.  Therefore, Lemma~\ref{lem:CracFein_subnetwork} implies that the CFSTR $\Net=\tilde G$ fails the Jacobian Criterion.  This is consistent with the computation of the Jacobian determinant~\eqref{eq:ex_Jac} in Example~\ref{ex:Jac}, which we saw contains one negative term.
\end{example}

\subsection{Reactant and reaction rank}
The Jacobian Criterion requires that each $s$-square subnetwork $G_i$ of $\Net$ has either a singular reactant matrix, a singular reaction matrix, or both matrices have nonzero determinants with the same sign.  We now give a name to these cases:

\begin{definition} A square \CRN $N$ is said to have {\em full reactant-rank} 
if $\det(M_N) \neq 0$, and is said to have {\em full reaction-rank} if $\det(R_N) \neq 0$.  
\end{definition}
If a square network has both full reactant-rank and full reaction-rank, then it clearly has {\em nonzero orientation}.  

\begin{definition}
In a square network $N$, we say that species $X_i$ is {\em changed} in reaction $j$ if $(R_N)_{ji} \neq 0$. 
\end{definition}

In other words, species $X_i$ is  changed in reaction $j$ if the concentration of species $X_i$ is changed when reaction $j$ takes place.  For example, species $A$ is changed in the reaction $A\rightarrow B+C$ and in the reaction $A+B \rightarrow 2A+B$.  On the other hand, species $B$ and $C$ are  changed only in the first reaction.

\begin{lemma} \label{lem:full_rk}
If $N$ is a square network with full reactant-rank, then:
\been
\item Each species of $N$ is a reactant species.
\item No reaction of $N$ is an inflow reaction or a generalized inflow reaction of the form $0 \rightarrow \sum_i a_i X_i $ (where $a_i \geq 0$).
\item No two reactions of $N$ share the same reactant complex.
\enen
If $N$ is an $n$-square network with full reaction-rank, then:
\been
\item Each species of $N$ is changed in some reaction.
\item $N$ does not contain both a reaction and its reverse.
\enen
\end{lemma}
\begin{proof} This result follows because nonsingular square matrices can not have a zero column, a zero row, two identical rows, or a row and its negative. \end{proof}

\subsection{Known tests for passing the Jacobian Criterion}
Known sufficient conditions on the reaction network for passing the Jacobian Criterion arise from the following results: the SCL Graph Theorem of Schlosser and Feinberg \cite{SchlosserFeinberg}, 
the Species-Reaction Graph Theorem of Craciun and Feinberg \cite{ME2, ME3} and the related work of Banaji and Craciun \cite{BanajiCraciun2009,BanajiCraciun2010}, and results of Helton, Klep, and Gomez \cite{HeltonDeterminant}.  In this article, we present a complementary test (Algorithm~\ref{algor:JC}).

\section{Analysis of subnetworks} \label{sec:analysis}

Theorems~\ref{thm:proper_subnetworks} and~\ref{thm:reduction} in this subsection form the basis of the procedure (Algorithm~\ref{algor:JC}) that we will introduce to simplify the implementation of the Jacobian Criterion.  The following example motivates our results.

\begin{example} \label{ex:1again}
We return to the network $G$ in Example~\ref{ex:1}.  Note that in both the reactant and product matrices~\eqref{matrices_ex} of the subnetwork $G_2$, the first row is the vector $(1,0)$.  So, the orientation of $G_2$ depends only on the entries in the lower-right of the two matrices.  In other words, by letting $N_2$ denote the 1-square network obtained from $G_2$ by `removing' the outflow reaction $A \ra 0$ and `removing' the species $A$ from the non-flow reaction $A+B \ra 2A$:
	\begin{align*} N_2 \quad = \quad \{ B \ra 0 \}~, \end{align*} 
it follows that $\Or(G_2)=\Or(N_2)$.  Similarly, for the network obtained from $G_3$ by removing $B \ra 0$ and the species $B$ from the non-flow reaction:
	\begin{align*} N_3 \quad = \quad \{ A \ra 2A \}~, \end{align*} 
we have $\Or(G_3)=\Or(N_3)$.  
\end{example}

In this section, we make the simplifications performed in Example~\ref{ex:1again} formal: the first rows in the matrices $M_{G_2}$ and $R_{G_2}$ \eqref{matrices_ex} will be called `pivotal rows' (Definition~\ref{def:pivot}), Theorem~\ref{thm:proper_subnetworks} will permit us to restrict our attention to the square `embedded' networks $N_i$ rather than the larger $s$-square subnetworks $G_i$, and Theorem~\ref{thm:reduction} will allow us to remove the species $B$ from the network.  In general the networks ${N}_i$ will be smaller and fewer in number, 
thereby simplifying the Jacobian Criterion.

\subsection{Analysis of row and column pivots} \label{sec:pivots}\quad
\begin{definition} \label{def:pivot}
\been
	\item For any matrix $A$, if there exists a row $i$ with one positive entry, $A_{ij}>0$ (respectively, $<0$) for some $j$, and all other entries zero, $A_{ik}=0$ for $k \neq j$, we say that $A_{ij}$ is a positive (respectively, negative) {\em row pivot} and the $i$-th row is a {\em pivotal row}. 
	\item For any matrix $A$, if  $A_{ij}>0$ (respectively, $<0$) for some $i$ and $A_{kj}=0$ for all $k \neq i$, we say that $A_{ij}$ is a positive (respectively, negative) {\em  column pivot} and the $j$-th column is a {\em pivotal column}.
\enen 
\end{definition}

\begin{notation}
For an $s  \times s$-matrix $A$ and indices $1 \leq i,j \leq s$, we define the associated matrix $A^{\widehat{ij}}$ which is obtained from $A$ by removing the $i$-th row and the $j$-th column.
\end{notation}

In the following theorem, we translate pivot conditions for reactant and reaction matrices into the language of reaction networks. Note that reactant matrices $M$ are non-negative matrices, so their pivots are necessarily positive pivots.
\begin{proposition} \label{prop:pivot}
Let $N$ be a square \CRN with reactant matrix $M$ 
and reaction matrix $R$.  Then:
\begin{enumerate}
	\item $M_{ij}$ is a row pivot if and only if the $i$-th reactant contains only species $j$ (in other words, the $i$-th reaction is $kX_j \rightarrow \cdot$).
	\item $M_{ij}$ is a column pivot if and only if species $j$ appears in the reactant complex of the $i$-th reaction and in no other reactant complex. 
	\item $R_{ij}$ is a row pivot if and only if only species $j$ is changed in reaction $i$.
	\item $R_{ij}$ is a column pivot if and only if species $j$ is changed only in reaction $i$. 
	\item Both $M_{ij}$ and $R_{ij}$ are row pivots if and only if the $i$-th reaction is of the form 
	\bd
	a_jX_j \rightarrow b_jX_j  ~~~~~~~~~{\text where}~~~ 0 \neq a_j\neq b_j
	\ed
Moreover, $R_{ij}$ is a positive pivot if $a_j > b_j$ and a negative pivot if $a_j < b_j$. If $M_{ij}$ and $R_{ij}$ are positive row pivots then we will 
say that $i$ is a {\em pivotal reaction}. 
	\item  Both $M_{ij}$ and $R_{ij}$ are column pivots if and only if species $X_i$ appears only in reaction $j$.  Moreover, $R_{ij}$ is a negative column pivot if and only if $X_i$ is a self-catalyst in reaction $j$.
	\item Suppose that for some $i,j$, both $M_{ij}$ and $R_{ij}$ are positive pivots. Then 
		\begin{align*}
		\Or(N) \quad = \quad 
			\sign(\det(M)\det(R)) \quad = \quad \sign(\det(M^{\widehat{ij}})\det(R^{\widehat{ij}}))~.
		\end{align*}
		 If $M_{ij}$ is a (positive) pivot and $R_{ij}$ is a negative pivot, then 
\begin{align*}
 \Or(N) \quad = \quad \sign(\det(M)\det(R)) \quad = \quad -\sign(\det(M^{\widehat{ij}})\det(R^{\widehat{ij}}))~.
\end{align*}

\end{enumerate}
\end{proposition}
\begin{proof}
This proposition follows easily from definitions and elementary linear algebra.
\end{proof}

\subsection{Analysis of embedded networks}
Theorem~\ref{thm:proper_subnetworks} in this section will tell us that instead of analyzing all $s$-square subnetworks $G_i$ of a network $\Net$, we can remove the outflow reactions from $G_i$ and then remove the outflow species from the remaining reactions, and analyze those simpler `embedded networks.' First we must define what is meant by removing a reaction or a species.
\begin{definition} \label{def:reduced_network}
Let $\Net=\{\SS,\CC,\RR\}$ be a \CRNNoSpace.
\begin{enumerate}
	\item Consider a subset of the species $S\subset \SS$, a subset of the complexes $C \subset \CC$, and a subset of the reactions $R\subset \RR$. 
\beit
\item 
The {\em restriction of $R$ to $S$}, denoted by $R|_S$, is the multiset of reactions obtained by taking the reactions in $R$ and removing all species not in $S$. 
\item The {\em restriction of $C$ to $R$}, denoted by $C|_R$, is the set of (reactant and product) complexes of the reactions in $R$.  
\item The {\em restriction of $S$ to $C$}, denoted by $S|_C$, is the set of species that are in the complexes in $C$. 
\enit
	
\item The network obtained from $\Net$ by {\em removing a subset of species} $\{X_i\} \subset \SS$ is the network 
		$$\left\{\SS\setminus \{X_i\} ,~\CC|_{\RR|_{\SS\setminus \{X_i\}}},~
			\RR|_{\SS\setminus \{X_i\}}  \right\}~.$$
	\item The network obtained from $\Net$ by {\em removing a set of reactions} $\{y \ra y'\} \subset \RR$ is the subnetwork 
		$$\left\{\SS|_{\CC|_{\RR \setminus \{y \ra y' \}}},~\CC|_{\RR \setminus \{y \ra y' \}},~\RR \setminus \{y \ra y' \} \right\}~.$$
\end{enumerate}
\end{definition}
\noindent
Note that $R|_S$ is viewed as a multiset (a generalized set in which an element may appear more than once), so that in the following definition, embedded networks are indeed square.

\begin{definition} \label{def:embeddednet}
	Let $\Net=\{\SS,\CC,\RR\}$ be a \CRN with $s$ species, and let $G$ be its non-flow subnetwork.~ 
	For $k=0,1,\dots,s$, a {\em $k$-square embedded network} of $\Net$ (or of $G$), which is defined by two sets, 
a $k$-subset of the species, $S=\left\{X_{i_1}, X_{i_2},\dots, X_{i_{k}} \right\}\subset \SS$,
and 
a $k$-subset of the non-flow reactions, $R=$ $\left\{ R_{j_1}, R_{j_2},\dots, R_{j_k} \right\}\subset$ $\RR$, that involve all species of $S$, 
is the $k$-square network $(S,\CC|_{R|_S},R|_S)$ consisting of the $k$ reactions $R|_S$. 
\end{definition}

The $1$-square embedded networks with negative orientation are precisely the reactions of the form $a X_j \rightarrow b X_j$ with $ 1 \leq a < b$.  Therefore, a network contains a negatively oriented $1$-square embedded reaction if and only if it contains a self-catalyzing reaction.

The following proposition will be used in the proofs of Theorems~\ref{thm:proper_subnetworks} and~\ref{thm:reduction}.
\begin{proposition} \label{prop:look_inside}
Let $N$ be a square network that contains one of the following:
	\begin{enumerate}
		\item an outflow reaction $X_i \ra 0$ or a generalized outflow reaction $a X_i \ra b X_i$ (where $0 \leq b < a$), or
		\item a species that appears in exactly one reaction, and the reaction is not self-catalyzing.
	\end{enumerate}
Then $\Or(N)=0$, or the network $\Nred$ obtained from $N$ by removing the aforementioned reaction and species is a square network and $\Or(N)=\Or(\Nred)$.
\end{proposition}
\begin{proof}
Part 1 follows easily from Proposition~\ref{prop:pivot}, as such a reaction is a pivotal reaction.  Now, part 2 is the case when some species $X_j$ appears in exactly one reaction $y \ra y'$ in $N$. If $X_j$ is in the product complex $y'$, then $\Or(N)=0$ by Lemma~\ref{lem:full_rk}. In the remaining case ($X_j$ is in the reactant complex $y$), then $y \ra y'$ is a pivotal reaction, so  $\Or(N)=\Or(\Nred)$.
\end{proof}

The following lemma generalizes what we saw in Example~\ref{ex:1again} in the CFSTR setting: each $s$-square subnetwork $G_i$ of $\Net$ defines an embedded network $N_i$ of $G$, and they have the same orientation.

\begin{lemma} \label{lem:embed_orientation} 
Let $\Net$ be a \CRN with $s$ species,
and let $G$ denote the subnetwork of non-flow reactions.
Consider the map from $s$-square subnetworks $G_i$ of $\Net$ that do not contain inflow reactions to square embedded networks of $G$, which maps $G_i$ to the network $N_i$ obtained from $G_i$ by first removing all outflow reactions $X_{i_l} \rightarrow 0$ (for $l=1,2,\dots,s-k$) and then removing those outflow species $X_{i_l}$ from the remaining non-flow reactions of $G_i$.
Then this map is surjective, and 
\begin{align} \label{eq:same_or}
\Or(G_i) \quad =\quad \Or(N_i)~.
\end{align}
(If $N_i$ is empty, its orientation is defined to be +1.)
\end{lemma}
\begin{proof}
The surjection is clear: both objects are defined by a subset of $\SS$ and a subset of $\RR_G$ whose cardinalities sum to $s$. Equation~\eqref{eq:same_or} follows from noting that the outflow reactions are pivotal reactions, and then applying Proposition~\ref{prop:look_inside}.
\end{proof}

From Lemma~\ref{lem:embed_orientation}, we obtain the following theorem, which will be used in Step~III in Algorithm~\ref{algor:JC}.
\begin{theorem} \label{thm:proper_subnetworks}
For a CFSTR $\Net$, with $G$ its subnetwork of non-flow reactions, the following are equivalent:
\begin{enumerate}
	\item $\Net$ passes the Jacobian Criterion.
	\item Each square embedded network of $G$ has non-negative orientation.
	\item $G$ has no self-catalyzing reaction, and each square embedded network $N$ of $G$ which satisfies the following properties has non-negative orientation:
	\begin{itemize}
		\item $N$ has no outflows, inflows, generalized inflow reactions $0 \rightarrow \sum_i a_i X_i$, or generalized outflows $a X_i \ra b X_i$ (where $0 \leq b \leq a$). 
		\item Each species appears in at least two reactions (where a pair of reversible reactions is counted only once).
	\end{itemize}
\end{enumerate}
\end{theorem}
\begin{proof}
Parts 1 and 2 are equivalent by Lemma~\ref{lem:CracFein_subnetwork} and Lemma~\ref{lem:embed_orientation}.  Clearly, part~2 implies part~3, so it remains only to show the converse.  It suffices to show that if $N$ is a square embedded network which violates one of the conditions of part~3, then either it has zero orientation or we can find a `smaller' embedded network which satisfies part 3 and has the same orientation as $N$.  In the case of inflows and generalized inflows, the orientation of $N$ is zero.  All remaining cases are handled by Proposition~\ref{prop:look_inside}.
\end{proof}

The following corollary follows easily from Theorem~\ref{thm:proper_subnetworks}.
\begin{corollary} \label{cor:1emb}
If a CFSTR contains a self-catalyzing reaction, then it fails the Jacobian Criterion.  Conversely, if a CFSTR does not contain a self-catalyzing reaction, then the only square embedded networks in Theorem~\ref{thm:proper_subnetworks} that must be checked are those of size at least two.
\end{corollary}

\noindent
Corollary~\ref{cor:1emb} allows us to see quickly by inspection that the network introduced in Example~\ref{ex:diffEq} (and analyzed in Examples~\ref{ex:Jac}, \ref{ex:1}, and~\ref{ex:1again}) fails the Jacobian Criterion.

\subsection{A network reduction theorem} \label{sec:reduction}
This subsection concerns those reactions or species that can be removed from a CFSTR without affecting whether the network passes the Jacobian Criterion.  The following theorem will be the foundation of Step~II of Algorithm~\ref{algor:JC}.
\begin{theorem} \label{thm:reduction}
Let $\Net$ be a CFSTR,  and let $\Net^{\rm red}$ be a network obtained from $\Net$ by performing any of the following operations:
\begin{enumerate}
	\item Remove any inflow or outflow reactions. 
	\item Remove any generalized inflow reactions $0 \rightarrow \sum_i a_i X_i $.
	\item Remove reactions containing only one species $a X_i \rightarrow b X_i$ for which $0 \leq b < a$.  
	\item Remove any species that is not a reactant species.
	\item Remove any species that appears in exactly one reaction (or a pair of reversible reactions), unless the species is a self-catalyst of the reaction.
\end{enumerate}
Then $\Net$ passes the Jacobian Criterion if and only if all square embedded networks of $\Net^{\rm red}$ have nonnegative orientation.
\end{theorem} 
\begin{proof}
Let $G$ (respectively, $\Gred$) denote the subnetwork of non-flow reactions of $\Net$ (respectively, $\Net^{\rm red}$).
By Theorem~\ref{thm:proper_subnetworks}, we need only show that all square embedded networks of $\Net$ have non-negative orientation if and only if all square embedded networks of $\Gred$ do.  
First, we consider the operations $1$ through $3$, in which $\Net^{\rm red}$ is obtained from $\Net$ by removing a reaction $y \ra y'$.  In this case, 
	\begin{align} \label{eq:rmvRxn}
	 \{ \text{SENs of } G \} ~ = ~  
		 \{ \text{ SENs of } G^{\rm red} \}  
		~ \cup ~ \{ \text{ SENs of } G \text{ that involve } y \ra y' \} ~,
	\end{align}
where `SEN' is shorthand for `square embedded network.'  
So, we need only show that if $N_i$ is a \sen of $G$, then either $N_i$ has zero orientation, or there exists a \sen $N^i_{\rm red}$ of $G^{\rm red}$ which has the same orientation as $N_i$.  In the case that the reaction $y \ra y'$ is an inflow or a generalized inflow, $\Or(N_i)=0$.  So, the remaining cases are when $y \ra y'$ is an outflow reaction $X_j \ra 0$ or is of the form $a X_j \rightarrow b X_j$  (where $1 \leq b < a$); this is handled by Proposition~\ref{prop:look_inside}. 

Now, only parts 4 and 5 remain to be proven.  Here, $\Net^{\rm red}$ is obtained from $\Net$ by removing a species $X_j$.  The analogue to~\eqref{eq:rmvRxn} now is that the set of \sens of $\Gred$ are precisely the \sens of $G$ that do not involve species $X_j$.  Again, we must show that a \sen $N_i$ of $G$ either has zero orientation, or there exists a \sen $N^i_{\rm red}$ of $G^{\rm red}$ with the same orientation as $N_i$.  If $X_j$ is not a reactant species of $G$ (in the case of part 4), then $N_i$ has zero orientation by Lemma~\ref{lem:full_rk}.  Finally, if $X_j$ satisfies the condition in part 5, then $X_j$ appears in exactly one reaction $y \ra y'$ in $N_i$.  This case is completed by applying Proposition~\ref{prop:look_inside}. This completes the proof.
\end{proof}

\section{Networks with total molecularity at most two do not admit multiple steady states} \label{sec:TM2}
In this section, we prove the following theorem, which concerns networks in which the maximum total molecularity (Definition~\ref{def:reversible_and_TM}) is at most two.  
\begin{theorem} \label{thm:total_mol_2}
If all species of a CFSTR have total molecularity less than or equal to two, then it passes the Jacobian Criterion, and therefore does not admit multiple steady states.  
\end{theorem}

Theorem~\ref{thm:total_mol_2} has the following corollary.
\begin{corollary} \label{cor:total_mol_2}
For any weakly-reversible \CRN (not necessarily a CFSTR) in which all species have total molecularity less than or equal to two, the network does not admit multiple steady states.
\end{corollary}

\begin{remark}
Theorem~\ref{thm:total_mol_2} cannot be improved to higher total molecularity.  For example, a CFSTR that contains the reaction $A\rightarrow 2A$ fails the Jacobian Criterion (by Corollary~\ref{cor:1emb}).
\end{remark}

\begin{remark}
Theorem~\ref{thm:total_mol_2} can be proven by applying the Species-Reaction Graph Theorem of Craciun and Feinberg~\cite[Theorem~1.1]{ME2}.  Rather than introducing the terminology and results of species-reaction graphs, we give a direct proof here.
\end{remark}

Our proof of Theorem~\ref{thm:total_mol_2} will proceed by showing that for a \CRN that has total molecularity at most two, if a square embedded network has nonzero orientation, then it in fact has positive orientation.  The proof requires the following lemmata, which will be used to reduce to studying  square embedded networks, and then to elucidate what values can be taken by the determinants of the reactant and reaction matrices of such networks.  

\begin{lemma} \label{lem:total_mol_embedded networks}
Let $\Net$ be a \CRN whose species all have total molecularity less than or equal to some $m\in\Z_{>0}$.  Then the same holds for any embedded network of $\Net$.
\end{lemma}

\noindent
The proof of Lemma~\ref{lem:total_mol_embedded networks} follows immediately from the definitions of total molecularity and embedded networks.
\begin{lemma} \label{lem:pseudo_at-most-2}
Let $N$ be a coupled $s$-square network with at least two reactions (none of which are inflows or outflows) such that all species have total molecularity at most two.  Assume that $N$ has nonzero orientation.  Then each reaction of $N$ has the form of $A+B \rightarrow 0$ or $A \rightarrow B$.  
\end{lemma}
\begin{proof}
We first claim that if $N$ is a coupled, $s$-square network in which all species have total molecularity two, then each reaction of $N$ must contain exactly two molecules.  This is because no reaction can have only one molecule (because inflows and outflows are not allowed).  So, each of the $s$ reactions must have exactly two molecules, and it follows that the total molecularity of each species must be two.  Now, a reaction of the form $0 \rightarrow A+B$ would yield zero orientation for the network (because the corresponding row of the reactant matrix would be zero), and a reaction $2A \rightarrow 0$ is not allowed, because then the network would decouple into this reaction and the remaining subnetwork that does not contain $A$.  So, only two types of reactions are possible: $A+B \rightarrow 0$ and $A \rightarrow B$.   
\end{proof}

\begin{lemma} \label{lem:det_pseudo}
Consider an $n \times n$-matrix $Q \in \{ -1,0,1\}^{n \times n}$ 
such that 
\begin{enumerate}
\item all diagonal entries are $+1$, 
\item the entries above the diagonal (mod $n$), $Q_{j,j+1 ({\normalfont mod}~ n)}$, are $0$ or $\pm 1$, and 
\item all other entries are 0.
\end{enumerate}
Then $\det(Q) \geq 0$. (In fact, the determinant is $0$, $1$, or $2$.)
\end{lemma} 
\begin{proof}
The determinant is $\det(Q)=1 \pm \prod_{i=1}^n (Q_{i,i+1({\normalfont mod}~ n)}) \in \{0,1,2\}$.
\end{proof}



We now can prove Theorem~\ref{thm:total_mol_2} and Corollary~\ref{cor:total_mol_2}.
\begin{proof}[Proof of Theorem~\ref{thm:total_mol_2}]
By Theorem~\ref{thm:proper_subnetworks}, Lemma~\ref{lem:embed_orientation}, and Lemma~\ref{lem:total_mol_embedded networks}, we need only show that if $N$ is a square network with nonzero orientation and with maximum total molecularity at most two, then $\det (M_N)\det (R_N) \geq 0$. Let $s$ denote the number of species of $N$.  Now, if $N$ decouples, then we could reduce to analyzing the two subnetworks separately both of which have fewer than $s$ species; the network $N$ passes the Jacobian Criterion if and only if the two subnetworks do.  The $s=1$ case is trivial (only $2A \rightarrow 0$ is allowed, and in this case both reaction and reactant matrices are the same, so their determinants are the same), so we may proceed by induction on $s$ in the case of a coupled network $N$. 

Now Lemma~\ref{lem:pseudo_at-most-2} applies, so we know that the reactions of $N$ are of the form $A+B \rightarrow 0$ and $A \ra B$.  By relabelling the species, we can ensure that the first reaction contains species $X_1$ and $X_2$, the second reaction contains $X_2$ and $X_3$, and so on.  Therefore, both the reactant matrix and the reaction matrix $R_N$ now have the form of the matrix in Lemma~\ref{lem:det_pseudo}. By the lemma, both determinants are non-negative, so the orientation of the network $N$ is either zero or positive.

Finally, Lemma~\ref{lem:CracFein_subnetwork} implies that a CFSTR network that passes the Jacobian Criterion does not admit multiple steady states.  
\end{proof}

\begin{proof}[Proof of Corollary~\ref{cor:total_mol_2}]
For a weakly-reversible network $G$ that has total molecularity at most two, we consider the CFSTR $\tilde G$ obtained by augmenting the network $G$ by all outflows.  Theorem~\ref{thm:total_mol_2} implies that $\tilde G$ passes the Jacobian Criterion.  Finally, results due to Craciun and Feinberg allow us to conclude that the such a network precludes multiple steady states \cite[Theorem~8.2 and Corollary~8.3]{ME3}.
 \end{proof}


\section{Procedure for simplifying the Jacobian Criterion} \label{sec:algor}
Algorithm~\ref{algor:JC} is the main result of this article.  It combines the results in previous sections into a procedure that simplifies the implementation of the Jacobian Criterion.  After pre-processing in Step~I, the algorithm reduces a \CRN to a simpler network in Step~II, and then decreases the number of embedded networks of the reduced network that must be examined in Step~III.

\noindent
\begin{algorithm}[Simplifying the Jacobian Criterion] \label{algor:JC}
~\\
{\bf Step~I (Pre-processing step).}\\
{\bf Input:} a CFSTR $\Net$. \\
{\bf Output:} either whether $\Net$ passes the Jacobian Criterion, or $\Net$ (to be passed to Step~II).
\begin{enumerate}
	\item If $\Net$ has a self-catalyzing reaction (an embedded reaction $a X_j \rightarrow b X_j$ with $1 \leq a < b$), then $\Net$ automatically fails the Jacobian Criterion.
	\item If each species of $\Net$ has total molecularity at most two, then $\Net$ automatically passes the Jacobian Criterion.
	\end{enumerate}
If neither applies, then pass $\Net$ to Step~II. \\
%
{\bf Step~II (Reduction step).} \\
{\bf Input:} a CFSTR $\Net$ that passes Step~I. \\
{\bf Output:} a non-flow \CRN $G_{\rm red}$ such that the CFSTR $\tilde{G}_{\rm red}$ passes the Jacobian Criterion if and only if $\Net$ does.\\
Perform the following operations on the network $\Net$, in any order, as long as any one can be completed.
	\begin{enumerate}
	\item Remove any inflow or outflow reaction. 
	\item Remove any generalized inflow reaction $0 \rightarrow \sum_i a_i X_i $.
	\item Remove any reaction that contains only one species.  
	\item Remove any species that is not a reactant species.
	\item Remove any species that appears in only one reaction (or in a pair of reversible reactions and in no other reactions). In particular, remove any species whose total molecularity is one.
	\item Remove one copy of any repeated reactions.
\end{enumerate}
If the reduced network is the empty network or all species in the reduced network have total molecularity at most two, then $\Net$ passes the Jacobian Criterion.  Otherwise, pass the reduced network $G^{\rm red}$ to Step~III.\\
%
{\bf Step~III (Analysis of embedded networks step).} \\
{\bf Input:} a \CRN $G$ (such as the output of Step~II) which passes Step~I. \\
{\bf Output:} whether $\tilde G$ passes the Jacobian Criterion (if $G$ is the output of Step~II, then this step will determine whether the input $\Net$ of Step~II passes the Jacobian Criterion). \\
Consider all coupled square embedded networks $N$ of $G$ with the following properties:
	\begin{enumerate}
	\item $N$ contains none of the following reactions: 
	\begin{itemize}
		\item inflow or outflow reactions,
		\item generalized inflow reactions $0 \rightarrow \sum_i a_i X_i $,
		\item reactions containing only one species, 
		\item both a reaction and its reverse reaction, 
		\item two reactions that share the same reactant complex.
	\end{itemize}
	\item $N$ has at least two species.
	\item Each species of $N$ is a reactant species.
	\item Each species appears in at least two reactions (where pairs of reversible reactions are counted only once). 
	\item At least one species has total molecularity three or more.
	\item $N$ has nonzero orientation.
	\end{enumerate}
Output: The network $\tilde G$ passes the Jacobian Criterion if and only if all such square embedded networks $N$ have nonnegative orientation.
\end{algorithm}
\begin{proof}[Proof of correctness of Algorithm~\ref{algor:JC}]
Part 1 of Step~I follows from Corollary~\ref{cor:1emb}, and part 2 follows from Theorem~\ref{thm:total_mol_2}.  
Step~II is valid by Theorem~\ref{thm:reduction}.  
Parts $1$ through $4$ of Step~III follow from Lemma~\ref{lem:full_rk}, Corollary~\ref{cor:1emb}, and Theorem~\ref{thm:proper_subnetworks}.  Part 5 of the Step~III follows from Theorem~\ref{thm:total_mol_2}, and part 6 simply avoids embedded networks with zero orientation.
\end{proof}
\begin{remark} Instead of applying Step~III to a reduced network that is the output of Step~II, one can apply other tests to check the Jacobian Criterion.  
For instance, one can apply the Species-Reaction Graph Criterion \cite{ME2} to the reduced network or directly compute the expansion of the Jacobian determinant using available software \cite{Toolbox, CRNsoftware, BioNetX}.  The reduced network typically contains fewer species and fewer reactions than the original network, so computer computation will be more tractable.
\end{remark}

\begin{remark} 
In practice, the number of square embedded networks that must be examined in Step~III
of Algorithm~\ref{algor:JC} is relatively small compared to the number of $s$-square subnetworks $G_i$ in Lemma~\ref{lem:CracFein_subnetwork}.  Comparisons of the number of square subnetworks $G_i$ versus the number of square embedded networks to be examined appear in the last two columns of Table~\ref{table:ME1} and in the comments of Table~\ref{table:CTF}.  In addition, the determinants of such square embedded networks (as given in Lemma~\ref{lem:CracFein_subnetwork0}) that must be computed are of smaller matrices than those of $G_i$.
\end{remark}

\begin{remark} \label{rem:enum}
 We propose that Algorithm~\ref{algor:JC} could aid in enumerating \CRNs that admit multiple steady states. Recall that failing the Jacobian Criterion is a necessary condition for admitting multiple steady states (Lemma~\ref{lem:CracFein_subnetwork}), so let us consider the problem of identifying the set of networks that fail the Jacobian Criterion. More specifically, we now define a {\em bimolecular} network to be a network in which all complexes contain at most two molecules, and then pose the following question: which bimolecular CFSTRs fail the Jacobian Criterion?

Accordingly, we enumerated all bimolecular CFSTRs that contain two pairs of reversible non-flow reactions and we found that there are 386 such networks. Of these 386 networks, 142 have a maximum total molecularity of two or fewer, and thus immediately pass the Jacobian Criterion. 
It is interesting to note that for these 142 networks, several distinct species typically have the same total molecularity. Networks that have such symmetry are computationally the most expensive to enumerate \cite{Deckard}. As we are interested in listing only those networks that fail the Jacobian Criterion, we can avoid these computationally expensive networks.  

Moreover, of the remaining 244 networks, applying Step~II of the Algorithm~\ref{algor:JC} immediately reduces the number of networks that need to be examined. In other words, many networks are related in whether or not they pass the Jacobian Criterion and thus studying all 244 networks is redundant. When we discard networks in which some species has total molecularity equal to one, we are left with only 55 networks whose multistationarity needs to be investigated.  Details of our enumeration of bimolecular CFSTRs appear in \cite{JS2}.  Note that a similar enumeration of small bimolecular reaction networks was undertaken by Deckard, Bergmann, and Sauro \cite{Deckard}, from which Pantea and Craciun sampled networks to compute the fraction of such networks that pass the Jacobian Criterion \cite[Figure 1]{Pantea_comput}.
\end{remark}

\section{Examples} \label{sec:examples}
The examples in this section illustrate that Algorithm~\ref{algor:JC} in some cases can determine whether a network passes the Jacobian Criterion more quickly than applying the Species-Reaction Graph Criterion of Craciun and Feinberg~\cite{ME2}, Deficiency Theory \cite{FeinLectures,FeinDefZeroOne}, or the results of Banaji and Craciun~\cite{BanajiCraciun2010}.  Moreover, we are able to analyze examples in~\cite{ME1,ME2,CTF06,SchlosserFeinberg}, which appear in Tables~\ref{table:ME1} and~\ref{table:CTF}.  In addition, we demonstrate that instead of checking the orientations of all the $s$-square networks $G_i$ from Lemma~\ref{lem:CracFein_subnetwork} (which entails computing determinants of many $s \times s$-matrices), we need only compute the orientations of certain square embedded networks, which are both smaller and fewer in number.  The final columns in both Tables~\ref{table:ME1} and~\ref{table:CTF} illustrate the great reduction in the number of subnetworks that must be examined.  We begin with examples in which no orientations must be computed, as checking the Jacobian Criterion sometimes can be performed by inspection.
\begin{table}
\begin{center}
  \begin{tabular}{ | l l | c | c | c | c | c |}
    \hline
	~ & ~ & ~ & ~ & ~ & ~ & ~ \\
	~ & ~ & 		$\tilde G$ & 	$\tilde G$ does & ~ & Number & Number \\
	Network  & ($G$) & 	passes  &   	not admit & Remarks & of $G_i$,  & of SENs \\ 
	~ & ~ & 		JC & 		MSS & ~ & $r+r_o \choose s$ & to examine \\  \hline
    (i)	& 	$A+B \lra P$ & ~ 	& ~ 	& ~ &~&~ \\ 
    ~	& 	$B+C \lra Q$ & No 	& No 	& See Ex.~\ref{ex:fewerEmbeddedNetworks} & ${{6+5} \choose 5} = 462 $  & 2\\ 
    ~	& 	$C \lra 2A$ & ~ 	& ~ 	& ~&~&~ \\ \hline
    (ii)& 	$A+B \lra P$ & ~ 	& ~ 	& ~&~&~ \\ 
    ~ 	& 	$B+C \lra Q$ & Yes 	& Yes 	& See Ex.~\ref{ex:fewerEmbeddedNetworks}  & ${{8+7} \choose 7} = 6435 $ & 2 \\ 
    ~	& 	$C+D \lra R$ & ~ 	& ~ 	& ~&~&~ \\
    ~	& 	$D \lra 2A$ & ~ 	& ~ 	& ~&~&~ \\ \hline
    (iii)& 	$A+B \lra P$ & ~ 	& ~ 	& ~&~&~ \\ 
    ~	& 	$B+C \lra Q$ & ~ 	& ~ 	& ~&~&~ \\
    ~	& 	$C+D \lra R$ & No 	& No 	& See Ex.~\ref{ex:fewerEmbeddedNetworks} & ${{10+9} \choose 9} = 92 378 $ & 2 \\ 
    ~	& 	$D+E \lra S$ & ~ 	& ~ 	& ~&~&~ \\ 
    ~	& 	$E \lra 2A$ & ~ 	& ~ 	& ~&~&~ \\ \hline
    (iv)&	$A+B \lra P$ & ~ 	& ~ 	& ~ &~&~ \\ 
    ~	& 	$B+C \lra Q$ & Yes 	& Yes 	& TM $\leq 2$,& ${{6+5} \choose 5} = 462 $ & {\bf --} \\ 
    ~	& 	$C \lra A$ & ~ 		& ~ 	& see Ex.~\ref{ex:TM2}	& ~ &~ \\ \hline
    (v)	& 	$A+B \lra F$ & ~ 	& ~ 	& ~&~&~ \\ 
    ~ 	& 	$A+C \lra G$ & No 	& No 	& See Ex.~\ref{ex:fewer_SENs}&  ${{8+7} \choose 7} = 6435 $ & 3 \\ 
    ~	& 	$C+D \lra B$ & ~ 	& ~ 	& ~&~&~ \\
    ~	& 	$C+E \lra D$ & ~ 	& ~ 	& ~&~&~ \\ \hline
    (vi)& 	$A+B \lra 2A$ & No 	& Yes 	& Self-catalyzing,&~&~  \\ 
	~ & ~ & ~ & ~ 				& see Ex.~\ref{ex:self-cat} & ${{2+2} \choose 2} = 6$ &{\bf --}\\ \hline
    (vii)& 	$2A+B \lra 3A$ & No 	& No 	& Self-catalyzing, &~&~ \\ 
	~ & ~ & ~ & ~ 				& see Ex.~\ref{ex:self-cat}& ${{2+2} \choose 2} = 6$ &{\bf --} \\ \hline
    (viii)& 	$A+2B \lra 3A$ & No 	& Yes 	& Self-catalyzing, &~&~ \\ 
	~ & ~ & ~ & ~ 				& see Ex.~\ref{ex:self-cat}& ${{2+2} \choose 2} = 6$ &{\bf --} \\ \hline
  \end{tabular}
\caption{Networks examined in Table~1 of Schlosser and Feinberg \cite{SchlosserFeinberg}.  Listed are the non-flow networks $G$, whether the CFSTR network $\tilde G$ passes the Jacobian Criterion (JC), whether $\tilde G$ precludes multiple steady states (MSS), additional comments, the number of $s$-square networks $G_i$ of $\tilde G$ that do not contain inflow reactions (here, $r_o = s$; see~Lemma~\ref{lem:CracFein_subnetwork}), and the number of square embedded networks (SENs) which must be considered in Algorithm~\ref{algor:JC}.  The entry in the last column is left blank ({\bf --}) for networks that can be determined to pass or to fail the Jacobian Criterion due to a self-catalyzing reaction or low total molecularity.  Note that the number of such SENs is much fewer than the number of subnetworks $G_i$, resulting in a substantial reduction in the number and size of the orientations that must be computed.  
Many of these networks were also analyzed by Craciun and Feinberg~\cite[Table 1.1]{ME1}, Banaji, Donnell, and Baigent \cite{BanajiDonBai} and Craciun, Helton, and Williams \cite{CHW08}.  
\label{table:ME1} }
\end{center} 
\end{table} 

\begin{table}
\begin{center}
  \begin{tabular}{ | l l | c | c | l | }
    \hline
	~ & ~ & 		~ & 		~ & 		~\\
	~ & ~ & 		$\tilde G$ & 	$\tilde G$ does & 	~\\
	Network & ($G$) & 	passes	&	not admit & 	Remarks \\
	~  & ~ &  JC  &  	MSS & 		~ \\ \hline
    (i)	& 	$E+S \lra ES \ra E+P$ 
				& Yes 	& Yes 	& Total molecularity   \\ 
	~ & ~ & ~ & ~ 				& $\leq 2$, see Ex.~\ref{ex:TM2} \\ \hline
    (ii)	& 	$E+S \lra ES \ra E+P$ 
				& Yes 	& Yes	& Reduced network has  \\ 
	~ & $E+I \lra EI$ & ~ & ~ 				& TM $\leq 2$, see Ex.~\ref{ex:reduceThenTM} \\ \hline
    (iii)	& 	$E+S \lra ES \ra E+P$ 
				& Yes 	& Yes	& Reduced network has   \\ 
	~ & $ES+I \lra ESI$ & ~ & ~ 				& TM $\leq 2$, see Ex.~\ref{ex:reduceThenTM} \\ \hline
    (iv)	& 	$E+S \lra ES \ra E+P$ 
				& ~	& ~ 	& Only 9 SENs  \\ 
	~ & $E+I \lra EI$ & No 		& No 				& to examine,  \\
	~ & $ES+I \lra ESI \lra EI+S $ & ~ & ~ 				& see Example~\ref{ex:fewer_SENs} \\ \hline
   (v)	& 	$E+S_1 \lra ES_1 $ 
				& Yes 	& Yes	& Total molecularity   \\ 
	~ & $S_2 +ES_1 \lra ES_1S_2 \ra E+P$ & ~ & ~ & $\leq 2$, see Ex.~\ref{ex:TM2} \\ \hline  
    (vi)	& 	$E+S_1 \lra ES_1 \quad E+S_2 \lra ES_2$ 
				& ~	& ~ 	& Only 14 SENs   \\ 
	~ & $S_2 +ES_1 \lra ES_1S_2 \lra S_1+ES_2$ 
				& No 	& No 				& to examine, \\
	~ & $ES_1S_2 \ra E+P$ & ~ & ~ 				& see Example~\ref{ex:fewer_SENs} \\ \hline
    (vii)	& 	$S_1+E_1 \lra E_1S_1$ 
				& ~	& ~ 	& Only 1 SEN  \\ 
	~ & $S_2 +E_1S_1 \lra E_1S_1S_2 \ra P_1+E_1$ & Yes & Yes 				& to examine, \\
	~ & $E_2+S_2 \lra E_2S_2 \ra 2S_1+E_2$ & ~ & ~ 				& see Example~\ref{ex:fewer_SENs} \\ \hline
    (viii)	&  $S_1+E_1 \lra E_1S_1$		& ~ & ~ 	& ~ \\ 
	~ & $S_2 +E_1S_1 \lra E_1S_1S_2 \ra P_1+E_1$ 	& ~ & ~ 	& Only 1 SEN \\ 
	~ & $S_2 +E_2 \lra E_2S_2$ 			& No & No 	& to examine, \\ 
	~ & $S_3 +E_2S_2 \lra E_2S_2S_3 \ra P_2 +E_2$ 	& ~ & ~ 	& see Example~\ref{ex:fewer_SENs} \\ 
	~ & $S_3 +E_3 \lra E_3S_3 \ra 2S_1 + E_3$ 	& ~ & ~ 	& ~ \\ \hline
    (ix)	&  $S_1+E_1 \lra E_1S_1$		& ~ & ~ 	& ~ \\
	~ & $S_2 +E_1S_1 \lra E_1S_1S_2 \ra P_1+E_1$  	& ~ & ~ 	& ~ \\
	~ & $S_2 +E_2 \lra E_2S_2$  			& ~ & ~		&  Only 1 SEN  \\
	~ & $S_3 +E_2S_2 \lra E_2S_2S_3 \ra P_2 +E_2$ 	& Yes & Yes	& to examine, \\
	~ & $S_3 +E_3 \lra E_3 S_3$ 			& ~ & ~ 	& see Example~\ref{ex:fewer_SENs} \\
	~ & $S_4 +E_3S_3 \lra E_3S_3S_4 \ra P_3 +E_3$ 	& ~ & ~		& ~ \\ 
	~ & $S_4 +E_4 \lra E_4S_4 \ra 2S_1 + E_4$ 	& ~ & ~ 	& ~ \\ \hline
\end{tabular}
\caption{The enzyme catalysis networks examined in Table~1 of Craciun, Tang, and Feinberg~\cite{CTF06}.  Listed are the underlying networks $G$, whether the CFSTR network $\tilde G$ passes the Jacobian Criterion (JC), whether $\tilde G$ precludes multiple steady states (MSS), and additional comments, including the number of square embedded networks (SENs) that need to be examined for those networks for which total molecularity (TM) is not sufficient for verifying the Jacobian Criterion.  \label{table:CTF} }
\end{center} 
\end{table} 

\begin{example}[Networks containing a self-catalyzing reaction] \label{ex:self-cat}
Recall that Corollary~\ref{cor:1emb} states that if a \CRN contains a self-catalyzing reaction, then it fails the Jacobian Criterion.  One such network is the one examined in Examples~\ref{ex:1} and~\ref{ex:1again}.  Other networks that are resolved easily by Step~I of Algorithm~\ref{algor:JC} include networks (vi), (vii), and (viii) in Table~\ref{table:ME1}.
\end{example}

In addition to the self-catalyzing networks, the other networks that are resolved in Step~I of Algorithm~\ref{algor:JC} are those with total molecularity at most two.
\begin{example}[Networks with total molecularity at most two] \label{ex:TM2}
Consider the following network examined by Gnacadja {\em et al.}\ \cite{Gnac}:
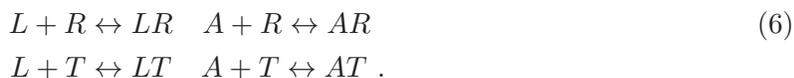
\begin{align}\label{net:LRAT}
	L+R & \leftrightarrow LR \quad  A+R \leftrightarrow AR \\
	L+ T  & \leftrightarrow L T  \quad  A+ T \leftrightarrow AT~. \notag
\end{align}
This ligand-receptor-antagonist-trap network summarizes the four pairs of binding and dissociating reactions that occur in the presence of two drugs, a `receptor antagonist' and a `trap,' which are used to treat patients with rheumatoid arthritis.  In their work, Gnacadja {\em et al.}\ apply Deficiency Theory to determine the existence and uniqueness of a steady state~\cite[Proposition~2]{Gnac}.  If one is  interested only in precluding multiple steady states, then the approach advocated here can be performed more quickly.  It is easily seen that the total molecularity of each of the eight species, $L,~R,~A,~T,~LR,~AR,~LT,$ and $AT$, is no more than two.  Therefore, Theorem~\ref{thm:total_mol_2} implies that a CFSTR version of this network passes the Jacobian Criterion, and does not admit multiple steady states.  Moreover, as the underlying network~\eqref{net:LRAT} is weakly-reversible, we can conclude that multiple steady states is also ruled out in the 
non-CFSTR setting \cite{ME3}.

Other networks with total molecularity at most two include the network (iv) in Table~\ref{table:ME1} and the enzyme catalysis networks (i) and (iv) in Table~\ref{table:CTF}.  All of these networks can be seen by inspection to pass the Jacobian Criterion, without appealing to the SR Graph Criterion.
\end{example}

The networks in Example~\ref{ex:TM2} had total molecularity at most two, and so could be seen immediately to pass the Jacobian Criterion.  The networks in the next example do not have this property until after applying Step~II of Algorithm~\ref{algor:JC}.
\begin{example}[Networks with total molecularity at most two after reduction] \label{ex:reduceThenTM}
Consider the CFSTR introduced by Banaji and Craciun whose non-flow subnetwork is:
	\begin{align} \label{net:BC}
	D  \lra A+B+C \quad \quad
	E  \lra A+B+C \quad \quad
	F  \lra A+B
	\end{align} 
Banaji and Craciun showed that the CFSTR fails the Species-Reaction Graph Criterion, but nevertheless passes the Jacobian Criterion, because the stoichiometric matrix is `strongly sign determined' ~\cite[\S 6.2]{BanajiCraciun2010}.  We now show that Step~II of Algorithm~\ref{algor:JC} can easily be performed on this network, to conclude that it passes the Jacobian Criterion. 

In Step~I of Algorithm~\ref{algor:JC}, we note that the total molecularity of species $A$ and $B$ in network~\eqref{net:BC} are both three, and no self-catalyzing reactions exist.  So, we begin Step~II by removing species $D$, $E$, and $F$, each of which has total molecularity equal to one.  The resulting network is:
	\begin{align*} 
	0  \lra A+B+C \quad \quad
	0  \lra A+B+C \quad \quad
	0  \lra A+B
	\end{align*} 
Now the first two pairs of reversible reactions are the same, so we remove one copy.  Additionally, we remove the remaining two generalized inflow reactions, to obtain:
	\begin{align*} 
	0  \la A+B+C \quad \quad
	0  \la A+B
	\end{align*} 
Clearly, each species in this reduced network has total molecularity no more than two, so we conclude that the original CFSTR passes the Jacobian Criterion.
Two other networks whose reduced networks can be obtained easily and have total molecularity at most two are the enzyme catalysis networks with inhibition (ii) and (iii) in Table~\ref{table:CTF}.
\end{example}

While the above examples did not need to proceed beyond Step~II of Algorithm~\ref{algor:JC}, the networks in the  following example do require this step.  We will see that the algorithm greatly reduces the number of subnetworks that must be checked.
\begin{example}[Networks requiring the full algorithm] \label{ex:fewerEmbeddedNetworks}
Consider the reversible network (i) in Table~\ref{table:ME1}; its non-flow subnetwork is: 
\begin{align}\label{net:CF_i}
 A+B \lra P \quad \quad 
B+C \lra Q \quad \quad 
C \lra 2A
\end{align} 
The output of Step~II of Algorithm~\ref{algor:JC} is the following network:
\begin{align*}
 A+B \rightarrow 0 \quad \quad 
B+C \rightarrow 0 \quad \quad 
C \lra 2A
\end{align*} 
As only species $A$ has total molecularity greater than two, each square embedded network we consider must contain both $A+B \rightarrow 0$ and one of the reversible reactions $C \lra 2A$ (because we need not consider networks with both directions of a reversible reaction).  However, such a $2$-square embedded network would have a species ($B$ or $C$) that appears in only one complex. Therefore, we need only consider the two $3$-square embedded networks, which are the following:
\begin{align} \label{net:emb1}
 \{ A+B \rightarrow 0 \quad \quad 
& B+C \rightarrow 0 \quad \quad 
C \ra 2A\} ~, \quad {\rm and} \\
 \{ A+B \rightarrow 0 \quad \quad 
& B+C \rightarrow 0 \quad \quad 
C \la 2A \} \label{net:emb2}
\end{align}
The first embedded network~\eqref{net:emb1} has negative orientation, and the second network~\eqref{net:emb2} has positive orientation, so a CFSTR version of the original network~\eqref{net:CF_i} fails the Jacobian Criterion.  Note that our algorithm reduced the Jacobian Criterion from examining the ${{6+5} \choose{5}}=462$ $s$-square subnetworks (the $G_i$ in Lemma~\ref{lem:CracFein_subnetwork}) to two embedded networks.  Networks~(ii),~(iii), and~(v) in Table~\ref{table:ME1} can be analyzed similarly: they require examining only two, two, and three square embedded networks, respectively.
\end{example}

Our final example network requires more complicated analysis, but the number of square embedded networks that must be checked is nevertheless small.
\begin{example} \label{ex:fewer_SENs}
Consider network (iv) in Table~\ref{table:CTF}. After removing species `$P$' in Step~II and then relabeling the six remaining species, the  subnetwork of non-flow reactions is:
\begin{align*}
A+B \lra C  \quad \quad
C \ra A  \quad \quad
A+E \lra F \quad \quad
C+E \lra G \quad \quad
G \lra F+B 
\end{align*}
Only one reaction, $C \ra A$, is not reversible; we label the remaining four reversible reactions as $R_1$ through $R_4$.  We now show that instead of checking that the ${16 \choose 7} =11440$ $7$-square subnetworks of the original CFSTR (with $7$ species and $9$ non-flow reactions) have non-negative orientation, Algorithm~\ref{algor:JC} allows us to examine only nine square embedded networks.  

Let $N$ denote a square embedded network that must be considered according to Step~III of Algorithm~\ref{algor:JC}.  The total molecularities of the species $(A,B,C,E,F,G)$ are respectively $(3,2,3,2,2,2)$. So species $A$ together with all its reactions ($R_1$,~$R_2$,~$R_3$) or species $C$ together with all its reactions ($R_1$,~$R_2$,~$R_4$) must be in $N$. In particular, both reactions $R_1$ and $R_2$ must appear in $N$, and both $A$ and $C$ must take part so that $R_2$ is not a flow reaction. Next, in order for species $C$ not to appear twice as a reactant, $R_1$ must appear in $N$ in the forward direction $A+B \ra C$. 

Now, the rows corresponding to reactions $R_1$ and $R_2$ in the reaction matrix of $N$ must be linearly independent, so the species $B$ must be in $N$. Thus, we also must include reaction $R_5$ so that $B$ has total molecularity greater than 1. Finally, so that $R_5$ is not a flow reaction, we must include species $F$ or $G$.

 Therefore, $N$ must be a 4-square or a 5-square embedded network, and the only possible 4-square networks are the following:
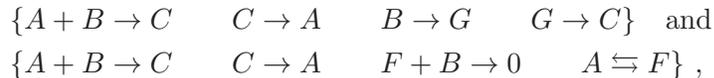
\begin{align*}
&\{A+B \ra C \quad \quad C \ra A \quad \quad B \ra G \quad \quad G \ra C\} \quad {\rm and} \\
&\{A+B \ra C \quad \quad C \ra A \quad \quad F+B \ra 0 \quad \quad A \lra F\}~,  
\end{align*}
where the second network represents two square embedded networks.  Additionally, the only possible 5-square networks are 
\begin{align*}
&\{A+B \ra C \quad \quad C \ra A \quad \quad A \ra F \quad \quad G \ra C \quad \quad F+B \ra G\}~, \\
&\{A+B \ra C \quad \quad C \ra A \quad \quad A + E \ra 0 \quad \quad C +E \lra G \quad \quad G \lra B\}~, \quad {\rm and} \\
&\{A+B \ra C\quad \quad C \ra A \quad \quad A + E \lra F \quad \quad C +E \ra 0 \quad \quad F+B \ra 0\}. 
\end{align*}
Here, the second network represents three networks (the network in which $G$ appears twice as a reactant is not permitted), and the third network represents two networks.  Therefore, there are 9 possible square embedded networks $N$.  Using similar analysis, network (v) of Table~\ref{table:ME1} and networks (vi), (vii), (viii), and (ix) of Table~\ref{table:CTF} can be shown to require checking only 3, 14, 1, 1, and 1 square embedded networks, respectively.

\end{example}



\subsection*{Acknowledgments}
This project was initiated by Badal Joshi at a Mathematical Biosciences Institute (MBI) summer workshop under the guidance of Gheorghe Craciun.  Badal Joshi was partially supported by a National Science Foundation grant (EF-1038593). Anne Shiu was supported by a National Science Foundation postdoctoral fellowship (DMS-1004380).  This work benefited from discussion with Murad Banaji. The authors acknowledge two conscientious referees whose detailed comments improved this article.
\bibliographystyle{amsplain}
\bibliography{multistationarity}

\providecommand{\bysame}{\leavevmode\hbox to3em{\hrulefill}\thinspace}
\providecommand{\MR}{\relax\ifhmode\unskip\space\fi MR }
\providecommand{\MRhref}[2]{%
  \href{http://www.ams.org/mathscinet-getitem?mr=#1}{#2}
}
\providecommand{\href}[2]{#2}
\begin{thebibliography}{10}

\bibitem{BanajiCraciun2009}
Murad Banaji and Gheorghe Craciun, \emph{Graph-theoretic approaches to
  injectivity and multiple equilibria in systems of interacting elements},
  Commun. Math. Sci. \textbf{7} (2009), no.~4, 867--900.

\bibitem{BanajiCraciun2010}
\bysame, \emph{{Graph-theoretic criteria for injectivity and unique equilibria
  in general chemical reaction systems}}, Adv. Appl. Math. \textbf{44} (2010),
  no.~2, 168--184.

\bibitem{BanajiDonBai}
Murad Banaji, Pete Donnell, and Stephen Baigent, \emph{{$P$} matrix properties,
  injectivity, and stability in chemical reaction systems}, SIAM J. Appl. Math.
  \textbf{67} (2007), no.~6, 1523--1547.

\bibitem{Conradi_subnetwork}
Carsten Conradi, Dietrich Flockerzi, J\"org Raisch, and J\"org Stelling,
  \emph{{Subnetwork analysis reveals dynamic features of complex (bio)chemical
  networks}}, Proc. Natl. Acad. Sci. USA \textbf{104} (2007), no.~49,
  19175--19180.

\bibitem{ME1}
Gheorghe Craciun and Martin Feinberg, \emph{Multiple equilibria in complex
  chemical reaction networks. {I}. {T}he injectivity property}, SIAM J. Appl.
  Math. \textbf{65} (2005), no.~5, 1526--1546.

\bibitem{ME_entrapped}
\bysame, \emph{Multiple equilibria in complex chemical reaction networks:
  extensions to entrapped species models}, IEE Proceedings-Systems Biology
  \textbf{153} (2006), 179--186.

\bibitem{ME2}
\bysame, \emph{Multiple equilibria in complex chemical reaction networks. {II}.
  {T}he species-reaction graph}, SIAM J. Appl. Math. \textbf{66} (2006), no.~4,
  1321--1338.

\bibitem{ME3}
\bysame, \emph{Multiple equilibria in complex chemical reaction networks:
  Semiopen mass action systems}, SIAM J. Appl. Math. \textbf{70} (2010), no.~6,
  1859--1877.

\bibitem{CHW08}
Gheorghe Craciun, J.~William Helton, and Ruth~J. Williams, \emph{Homotopy
  methods for counting reaction network equilibria}, Math. Biosci. \textbf{216}
  (2008), no.~2, 140--149.

\bibitem{CTF06}
Gheorghe Craciun, Yangzhong Tang, and Martin Feinberg, \emph{{Understanding
  bistability in complex enzyme-driven reaction networks}}, Proc. Natl. Acad.
  Sci. USA \textbf{103} (2006), no.~23, 8697--8702.

\bibitem{Deckard}
Anastasia~C. {Deckard}, Frank~T. {Bergmann}, and Herbert~M. {Sauro},
  \emph{{Enumeration and online library of mass-action reaction networks}},
  Available at {\tt arXiv/0901.3067}, 2009.

\bibitem{EllisonThesis}
Phillipp Ellison, \emph{The advanced deficiency algorithm and its applications
  to mechanism discrimination}, Ph.D. thesis, University of Rochester, 1998.

\bibitem{Toolbox}
Phillipp Ellison, Martin Feinberg, and Haixia Ji, \emph{{Chemical reaction
  network toolbox}}, Available at
  \url{http://www.che.eng.ohio-state.edu/~feinberg/crnt/}, 2011.

\bibitem{FeinLectures}
Martin Feinberg, \emph{{Lectures on chemical reaction networks}}, Notes of
  lectures given at the Mathematics Research Center of the University of
  Wisconsin in 1979, available at
  \url{http://www.che.eng.ohio-state.edu/~feinberg/LecturesOnReactionNetworks},
  1979.

\bibitem{FeinDefZeroOne}
\bysame, \emph{{Chemical reaction network structure and the stability of
  complex isothermal reactors I. The deficiency zero and deficiency one
  theorems}}, Chem. Eng. Sci. \textbf{42} (1987), no.~10, 2229--2268.

\bibitem{FW_prec}
Elisenda Feliu and Carsten Wiuf, \emph{{Preclusion of switch behavior in
  reaction networks with mass-action kinetics}}, Available at {\tt
  arXiv/1109.5149}, 2011.

\bibitem{Gnac}
Gilles Gnacadja, Alex Shoshitaishvili, Michael~J. Gresser, Brian Varnum, David
  Balaban, Mark Durst, Chris Vezina, and Yu~Li, \emph{Monotonicity of
  interleukin-1 receptor-ligand binding with respect to antagonist in the
  presence of decoy receptor}, J. Theoret. Biol. \textbf{244} (2007), no.~3,
  478--488.

\bibitem{HeltonDeterminant}
J.~William Helton, Igor Klep, and Raul Gomez, \emph{{Determinant expansions of
  signed matrices and of certain Jacobians}}, SIAM J. Matrix Anal. Appl.
  \textbf{31} (2009), no.~2, 732--754.

\bibitem{HornJackson72}
Fritz Horn and Roy Jackson, \emph{{General mass action kinetics}}, Arch.
  Ration. Mech. Anal. \textbf{47} (1972), no.~2, 81--116.

\bibitem{JS2}
Badal Joshi and Anne Shiu, \emph{Atoms of multistationarity in chemical
  reaction networks}, Available at {\tt arXiv:1108.5238}, 2011.

\bibitem{CRNsoftware}
Igor Klep, Karl Fredrickson, and Bill Helton, \emph{{Chemical reaction network
  software (under Mathematica)}}, Available at
  \url{http://www.math.ucsd.edu/~chemcomp/}, 2008.

\bibitem{BioNetX}
Casian Pantea, \emph{{BioNetX}}, Available at
  \url{http://www.math.wisc.edu/~pantea/}, 2010.

\bibitem{Pantea_comput}
Casian Pantea and Gheorghe Craciun, \emph{{Computational methods for analyzing
  bistability in biochemical reaction networks}}, Circuits and Systems (ISCAS),
  Proceedings of 2010 IEEE International Symposium on, IEEE, 2010,
  pp.~549--552.

\bibitem{TSS}
Mercedes {P\'erez Mill\'an}, Alicia Dickenstein, Anne Shiu, and Carsten
  Conradi, \emph{Chemical reaction systems with toric steady states}, To appear
  in B. Math. Biol., 2012.

\bibitem{SchlosserFeinberg}
Paul~M. Schlosser and Martin Feinberg, \emph{A theory of multiple steady states
  in isothermal homogeneous {CFSTR}s with many reactions}, Chem. Eng. Sci.
  \textbf{49} (1994), no.~11, 1749--1767.

\end{thebibliography}
\end{document}